\documentclass{article}

\usepackage{amsmath,amsfonts,amsthm,amssymb}  
\usepackage{dsfont}			
\usepackage{thmtools}		
\usepackage{enumerate}
\usepackage{graphicx}								

\usepackage[latin1]{inputenc}                       

\newtheorem{theorem}{Theorem}[section]
\newtheorem{proposition}[theorem]{Proposition}
\newtheorem{corollary}[theorem]{Corollary}

\newtheorem{remark}{Remark}

\theoremstyle{definition}
\newtheorem*{definition}{Definition} 

\newtheorem{example}[theorem]{Example}

\declaretheoremstyle[
spaceabove=6pt, spacebelow=6pt,
headfont=\normalfont\bfseries,
notefont=\normalfont\bfseries, 
notebraces={}{},
bodyfont=\normalfont\itshape
]{Estilo1}

\newcommand\Prod[1] 		{\langle #1 \rangle}

\def\R{\mathds{R}}
\def\C{\mathds{C}}
\def\I{\mathcal{I}}
\def\P{\mathds{P}}

\def\intprod		{\lrcorner\,}

\def\im	 				{\mathrm{i}}

\DeclareMathOperator{\Span}{span}

\DeclareMathOperator{\area}{area}

\begin{document}

\title{Projection factors and generalized real and complex Pythagorean theorems}

\author{Andr\'e L. G. Mandolesi 
               \thanks{Instituto de Matemática e Estatística, Universidade Federal da Bahia, Av. Adhemar de Barros s/n, 40170-110, Salvador - BA, Brazil. E-mail: \texttt{andre.mandolesi@ufba.br}}}

\date{\today.}

\maketitle

\abstract{Projection factors describe the contraction of Lebesgue measures in orthogonal projections between subspaces of a real or complex inner product space. They are connected to Grassmann's exterior algebra and the Grassmann angle between subspaces, and lead to generalized Pythagorean theorems, relating measures of subsets of real or complex subspaces and their orthogonal projections on certain families of subspaces. 

The complex Pythagorean theorems differ from the real ones in that the measures are not squared, and this may have important implications for quantum theory. Projection factors of the complex line of a quantum state with the eigenspaces of an observable give the corresponding quantum probabilities. The complex Pythagorean theorem for lines corresponds to the condition of unit total probability, and may provide a way to solve the probability problem of Everettian quantum mechanics.

\vspace{.5em}

}

\section{Introduction}

The importance of the Pythagorean theorem is reflected in its many generalizations. 
One which is well known, playing an essential role in quantum mechanics, is that if $X=\bigoplus_{i\in \I}V_i$ is a partition of a real or complex Hilbert space $X$ into a family $\{V_i\}_{i\in \I}$ of mutually orthogonal closed subspaces, and $P_i$ is the orthogonal projection on $V_i$, then for any $v\in X$,
\begin{equation}\label{eq:norm projections}
\|v\|^2 = \sum\limits_{i\in\I} \left\|P_i v \right\|^2.
\end{equation}

Less known generalizations, which keep being rediscovered time and again, relate areas, volumes or Lebesgue measures in general to their orthogonal projections \cite{Amir-Moez1996,Atzema2000,Conant1974,Donchian1935,Drucker2015}. They are usually proven for parallelotopes or simplices, using determinant identities such as the Cauchy-Binet formula \cite{Gantmacher2000}, being sometimes extended to other figures via integral calculus or measure theory. There is even a neat proof via divergence theorem \cite{Eifler2008}.

Using projection factors, which describe the contraction of Lebesgue measures under orthogonal projections, we reobtain these generalizations, with much simpler proofs, and extend them to complex vector spaces. 
This is done by relating these factors to Grassmann's exterior algebra \cite{Winitzki2010,Yokonuma1992}. Many useful properties are also obtained by connecting them to the Grassmann angle between subspaces \cite{Mandolesi_Grassmann,Mandolesi_Products}.

An unusual feature of the complex Pythagorean theorems is that measures are not squared (which is compensated by their dimensions being doubled), and this may allow projection factors to play a fundamental role in quantum mechanics. 
In the Hilbert space of a quantum system, the projection factors of the complex line of a quantum state with the eigenspaces of an observable equal the corresponding quantum probabilities. 
And the total probability being 1 corresponds to the complex Pythagorean theorem for lines. In \cite{Mandolesi_Born} we show this can be used to obtain the Born rule in Everettian quantum mechanics, i.e. the many-worlds interpretation \cite{EverettIII1957,Saunders2010a}.

In \autoref{sc:projection factors} we study projection factors, and relate them to the Grassmann algebra.
The generalized Pythagorean theorems are obtained in \autoref{sc:Pythagorean thms}, and \autoref{sc:epilogue} closes with a few remarks.
In \autoref{sc:Properties of projection factors} we list other useful properties of projection factors.

\section{Projection factors}\label{sc:projection factors}

Let $X$ be a $n$-dimensional vector space over $\R$ (real case) or $\C$ (complex case), with inner product  $\Prod{\cdot,\cdot}$ (Hermitian product in the complex case).
Unless otherwise indicated, whenever we refer to subspaces or other linear algebra concepts of $X$, it will be with respect to the same field as $X$. 

Complex vector spaces can also be seen as real ones, with twice the complex dimension, and the real part of a Hermitian product $\Prod{\cdot,\cdot}$ gives a real inner product $\operatorname{Re}\Prod{\cdot,\cdot}$ in the underlying real vector space. As $\C$-orthogonality (with respect to $\Prod{\cdot,\cdot}$) implies $\R$-orthogonality (with respect to $\operatorname{Re}\Prod{\cdot,\cdot}$), orthogonal projections with respect to both products coincide.

On any real or complex subspace $V\subset X$, let $|\cdot|_m$ be the $m$-dimensional Lebesgue measure if $m=\dim_\R V$, and $|\!\cdot\!|_m=0$ if  $m>\dim_\R V$. 

\begin{definition}
For subspaces $V,W\subset X$, let $P:V\rightarrow W$ be the orthogonal projection and $p=\dim_\R V$. The \emph{projection factor} of $V$ on $W$ is
\begin{equation*}\label{eq:projection factor}
\pi_{V,W}=\frac{|P(S)|_p}{|S|_p}, 
\end{equation*}
where $S$ is any Lebesgue measurable subset of $V$ with $|S|_p\neq 0$.
\end{definition}

\begin{remark}
$|P(S)|_p$ is defined, since $\dim_\R P(V)\leq p$. 
\end{remark}

\begin{remark}\label{rm:pi independent of S}
As $P$ is linear, $\pi_{V,W}$ does not depend on $S$. More precisely, $\pi_{V,W}=0$ if $\dim_\R P(V)<p$, and if $\dim_\R P(V)=p$ then $V$ and $P(V)$ can be isometrically identified with $\R^p$, in which case independence of $S$ is a defining property of the Lebesgue measure \cite[p.51]{Rudin1986}. 
\end{remark}

Clearly, the projection factor for complex subspaces equals that of their underlying real vector spaces.  Also, $\pi_{V,W}$ depends only on the relative position of $V$ and $W$, being invariant by any orthogonal transformation $T$ (unitary, in the complex case), as $T$ preserves Lebesgue measures and the orthogonal projection of $T(S)$ on $T(W)$ coincides with $T(P(S))$. 

Calculations will be simpler when $S$ is a parallelotope.

\begin{definition}
The \emph{parallelotope} spanned by $v_1,\ldots,v_k\in X$ is
\[ [v_1,\ldots,v_k] = \{t_1v_1+\ldots+t_kv_k : t_1,\ldots,t_k\in [0,1]\}. \]
\end{definition}

The parallelotope is \emph{orthogonal} if the $v_i$'s are mutually $\R$-orthogonal, in which case $\big| [v_1,\ldots,v_k]\big|_k = \|v_1\|\cdot\ldots\cdot\|v_k\|$.
If $P$ is the orthogonal projection onto a subspace $W\subset X$ then 
\[ P([v_1,\ldots,v_k]) = [Pv_1,\ldots,Pv_k]. \]

\subsection{Projection factors of lines}

\begin{definition}
For any $v\in X$, let $\R v=\{cv:c\in\R\}$, and in the complex case also $\C v=\{cv:c\in\C\}$.
\end{definition}

\begin{definition}
A \emph{line} $L\subset X$ is a 1-dimensional subspace. A nonzero $v\in X$ determines a \emph{real line} $\R v$ (in the complex case, it should be understood as a subspace of the underlying real vector space), and in the complex case also a \emph{complex line} $\C v$ (which is isometric to a real plane).
\end{definition}

\begin{proposition}\label{pr:Pv pi}
Given $v\in X$ and a real or complex subspace $W\subset X$, let $P$ be the orthogonal projection on $W$. Then
\begin{equation*}
\|Pv\|=\|v\|\cdot\pi_{\R v,W}.
\end{equation*}
In the complex case, with $W$ being a complex subspace, we also have
\begin{equation}\label{eq:projection complex line}
\|Pv\|^2=\|v\|^2\cdot\pi_{\C v,W}.
\end{equation}
\end{proposition}
\begin{proof}
	The first identity follows from the definition, and the other from the fact that orthogonal projections onto complex subspaces are $\C$-linear, so the square $[v,\im v]\subset\C v$ projects to the square $[Pv,\im Pv]$.
\end{proof}

In \autoref{sc:Properties of projection factors} we give more general formulas (\autoref{pr:properties pi}\,\ref{it:projection blade}).

\begin{example}\label{ex:quantum probability}
Let $X$ be the Hilbert space of a quantum system, $\C\psi$ be the complex line of a  quantum state $\psi$ (its \emph{ray}, except for the origin), and $W$ be the eigenspace of a quantum observable corresponding to some eigenvalue $\lambda$. By \eqref{eq:projection complex line}, $\pi_{\C\psi,W}$ equals the probability of obtaining $\lambda$ when measuring $\psi$ for that observable. In \cite{Mandolesi_Born} we explore this relation between projection factors and quantum probabilities.
\end{example}

\begin{proposition}\label{pr:product projection factor}
For any $v,w\in X$, we have:
\begin{itemize}
\item In the real case, $|\Prod{v,w}|=\|v\|\|w\|\cdot\pi_{\R v,\R w}$. 
\item In the complex case, 
\begin{align}
|\operatorname{Re}\Prod{v,w}|&=\|v\|\|w\|\cdot\pi_{\R v,\R w}, \nonumber \\
|\langle v,w\rangle| &=\|v\|\|w\|\cdot\pi_{\R v,\C w}, \nonumber \\
|\langle v,w\rangle|^2 &=\|v\|^2\|w\|^2\cdot\pi_{\C v,\C w}. \label{eq:Hermitian product projection factor}
\end{align}
\end{itemize}
\end{proposition}
\begin{proof}
Follows from \autoref{pr:Pv pi}, as the orthogonal projection of $v$ on $W=\R w$ (real case, or underlying real space in the complex one) or $W=\C w$ (complex case) is $Pv=\frac{\langle w,v \rangle}{\|w\|^2} w$.
\end{proof}

Projection factors between lines clearly depend on the angle between them. But some care is needed in the complex case, as there are different concepts of angle between complex vectors  \cite{Galantai2006,Scharnhorst2001}.

\begin{definition}
The \emph{Euclidean angle} $\theta_{v,w}\in[0,\pi]$ of nonzero vectors $v,w\in X$ is the usual angle between them, considered as real vectors, i.e.
\begin{equation*}
\cos\theta_{v,w} = \frac{\operatorname{Re}\langle v,w \rangle}{\|v\| \|w\|}. 
\end{equation*}
In the complex case, the \emph{Hermitian angle} $\gamma_{v,w}\in[0,\frac{\pi}{2}]$ is defined by
\begin{equation*}
\cos \gamma_{v,w} = \dfrac{|\langle v,w\rangle|}{\|v\|\|w\|}.
\end{equation*}
\end{definition}

The Hermitian angle is the (Euclidean) angle between $v$ and the plane $\C w$, and gives the Fubini-Study distance between the lines $\C v$ and $\C w$ in the complex projective space $\P(X)$.

\begin{corollary}\label{cr:projection factors and angles}
For any nonzero vectors $v,w\in X$, 
\[ \pi_{\R v,\R w}=\left|\cos\theta_{v,w}\right|.\]
Also, in the complex case, 
\begin{align*}
\pi_{\R v,\C w} &=\cos\gamma_{v,w}, \\
\pi_{\C v,\C w} &=\cos^2\gamma_{v,w}.
\end{align*}
\end{corollary}

\begin{example}
By \eqref{eq:Hermitian product projection factor}, the fidelity \cite{Bengtsson2017} of pure quantum states $\psi$ and $\varphi$ equals $\pi_{\C\psi,\C\varphi}$, and their Bures angle is the same as  $\gamma_{\psi,\varphi}$.
\end{example}

\subsection{Principal projection factors}

Let $V,W\subset X$ be nonzero subspaces, $p=\dim V$, $q=\dim W$, and $m=\min\{p,q\}$.
A singular value decomposition gives \cite{Galantai2006} orthonormal \emph{principal bases} $(e_1,\ldots,e_p)$ of $V$ and $(f_1,\ldots,f_q)$ of $W$ in which the orthogonal projection $P:V\rightarrow W$ is given by a $q\times p$ diagonal matrix $\mathbf{P}$ with real non-negative diagonal elements $1\geq\sigma_1\geq\ldots\geq\sigma_m\geq 0$, its \emph{singular values}. The $e_i$'s and $f_j$'s are \emph{principal vectors} of $V$ and $W$, and satisfy
\begin{equation}\label{eq:eifj}
\Prod{e_i,f_j}=\begin{cases}
0 \hspace{7pt} \text{ if } i\neq j, \\
\sigma_i \ \text{ if } i=j.
\end{cases}
\end{equation}

This has a nice geometric interpretation. The unit sphere of $V$ projects to an ellipsoid in $W$.  In the real case, for $1\leq i\leq m$, the $\sigma_i$'s are the lengths of its semi-axes, the $e_i$'s project onto them, and the $f_i$'s point along them. In the complex case, for each $1\leq i\leq m$ there are two semi-axes of length $\sigma_i$, corresponding to the projections of $e_i$ and $\im e_i$.

It also admits a variational characterization, given recursively as follows. For each $i=1,\ldots,m$, let
\begin{equation*}
\begin{cases}
V_i=\{v\in V: \Prod{v,e_j}=0\ \forall\, j<i\}, \\
W_i=\{w\in W: \Prod{w,f_j}=0\ \forall\, j<i\}, \\
\sigma_i = \max\left\{\,|\Prod{v,w}| : v\in V_i, w\in W_i, \|v\|=\|w\|=1 \right\}, \\
e_i\in V_i, f_i\in W_i \text{ are unit vectors such that } \Prod{e_i,f_i}=\sigma_i.
\end{cases}
\end{equation*}

\begin{definition}
For $1\leq i\leq m$, the \emph{principal projection factor} $\pi_i$ of $V$ and $W$ is the projection factor for the lines of $e_i$ and $f_i$, i.e.
\[ \pi_i =
\begin{cases}
\pi_{\R e_i,\R f_i} \,\text{ in the real case,} \\
\pi_{\C e_i,\C f_i} \  \text{ in the complex case.}
\end{cases} \]
\end{definition}

From \autoref{pr:product projection factor} and \eqref{eq:eifj} we get, for $1\leq i\leq m$,
\begin{equation}\label{eq:pi sigma}
\pi_i =
\begin{cases}
\sigma_i \ \,\text{ in the real case,} \\
\sigma_i^2 \  \text{ in the complex case.}
\end{cases}
\end{equation}

\begin{theorem}\label{pr:projection factors}
For nonzero subspaces $V,W\subset X$,
\[ \pi_{V,W}=\begin{cases}
\pi_1\cdot\ldots\cdot\pi_p \ \text{ if } \dim V=p\leq\dim W, \\
0 \hspace{15mm} \text{ if } \dim V>\dim W,
\end{cases} \]
where the $\pi_i$'s are their principal projection factors.
\end{theorem}
\begin{proof}
If $\dim V>\dim W$ then $|P(S)|_p=0$ by definition. Otherwise,
\begin{itemize}
\item in the real case, $S=[e_1,\ldots,e_p]$ is a hypercube with $|S|_p=1$ and $P(S)=[\sigma_1f_1,\ldots,\sigma_p f_p]$ is an orthogonal parallelotope with $|P(S)|_p=\sigma_1\cdot\ldots\cdot\sigma_p=\pi_1\cdot\ldots\cdot\pi_p$;
\item in the complex case, $S=[e_1,\im e_1,\ldots,e_p,\im e_p]$ is a hypercube with $|S|_{2p}=1$ and $P(S)=[\sigma_1f_1,\im \sigma_1 f_1,\ldots,\sigma_pf_p,\im \sigma_p f_p]$ is an orthogonal parallelotope with $|P(S)|_{2p}=\sigma^2_1\cdot\ldots\cdot\sigma^2_p=\pi_1\cdot\ldots\cdot\pi_p$. \qedhere
\end{itemize}
\end{proof}

So, in the real case, the factor by which $p$-dimensional volumes in $V$ shrink when projecting onto $W$ is the product of the factors by which lengths in the \emph{principal lines} $\R e_i$ contract.
In the complex case, the factor by which $2p$-dimensional volumes in $V$ shrink is the product of the factors by which areas in the \emph{principal complex lines} $\C e_i$ contract.

We now give practical formulas for computing $\pi_{V,W}$ in terms of orthonormal bases. In \autoref{pr:formula general bases} we extend them to other bases.

\begin{proposition}\label{pr:formula orthonormal bases}
Let $V,W\subset X$ be nonzero subspaces, and $\mathbf{P}$ be a matrix representing the orthogonal projection $P:V\rightarrow W$ in orthonormal bases of $V$ and $W$. Then
\[ 
\pi_{V,W}=
\begin{cases}
\sqrt{\det\left(\mathbf{P}^T \mathbf{P}\right)}  \,\text{ in the real case,} \\
\hspace{6pt} \det\left(\bar{\mathbf{P}}^T \mathbf{P}\right) \hspace{2pt} \text{ in the complex case.}
\end{cases} 
\]
If $\dim V=\dim W$ then 
\[ \pi_{V,W}=
\begin{cases}
|\det \mathbf{P}| \hspace{4pt}\text{ in the real case,} \\
|\det \mathbf{P}|^2 \text{ in the complex case.}
\end{cases} \]
\end{proposition}
\begin{proof}
It is enough to consider principal bases of $V$ and $W$, for which $\mathbf{P}$ is a diagonal matrix with the $\sigma_i$'s in the diagonal. Then the result follows from \eqref{eq:pi sigma} and \autoref{pr:projection factors}. 
\end{proof}

Principal projection factors are directly connected to Jordan's principal or canonical angles \cite{Galantai2006,Gluck1967,Wedin1983}. 

\begin{definition}
The \emph{principal angles} $0\leq \theta_1\leq\ldots\leq\theta_m\leq\frac{\pi}{2}$ of $V$ and $W$ are the angles $\theta_i=\theta_{e_i,f_i}$ between their principal vectors.
\end{definition}

As $\Prod{e_i,f_i}\geq 0$, in the complex case we also have $\theta_i=\gamma_{e_i,f_i}$.

By \eqref{eq:eifj}, $\sigma_i=\cos\theta_i$, so that, for $1\leq i\leq m$,
\begin{equation}\label{eq:principal factor angle}
\pi_i =
\begin{cases}
\cos\theta_i \ \,\text{ in the real case,} \\
\cos^2\theta_i \text{ in the complex case.}
\end{cases} 
\end{equation}

\subsection{Grassmann algebra and Grassmann angle}

The full power of projection factors is unleashed by connecting them to Grassmann's exterior algebra \cite{Suetin1989,Winitzki2010,Yokonuma1992}. 

A \emph{($p$-)blade} is a decomposable $p$-vector $\nu=v_1\wedge\ldots\wedge v_p\in\Lambda^p X$. 
The inner product (Hermitian, in the complex case) in the exterior power $\Lambda^p X$ is defined by extending linearly (sesquilinearly, in the complex case) the following formula for $p$-blades $\nu=v_1\wedge\ldots\wedge v_p$ and $\omega=w_1\wedge\ldots\wedge w_p$:
\begin{equation*}
\Prod{\nu,\omega} = \det\big(\Prod{v_i,w_j}\big) = \begin{vmatrix}
\Prod{v_1,w_1}& \cdots & \Prod{v_1,w_p} \\ 
\vdots & \ddots & \vdots \\ 
\Prod{v_p,w_1} & \cdots & \Prod{v_p,w_p}
\end{vmatrix}.
\end{equation*}

The norm $\|\nu\|=\sqrt{\Prod{\nu,\nu}}$ of a $p$-blade $\nu=v_1\wedge\ldots\wedge v_p$ gives, in the real case, the $p$-dimensional volume of $[v_1,\ldots,v_p]$. In the complex case, its square gives the $2p$-dimensional volume of $[v_1,\im v_1,\ldots,v_p, \im v_p]$. 

\begin{theorem}\label{pr:pi blades}
If $\dim V=\dim W=p$ then, for any nonzero blades $\nu\in \Lambda^p V$ and $\omega\in \Lambda^p W$,
\[ \pi_{V,W} = \begin{cases}
\hspace{2pt} \frac{|\Prod{\nu,\omega}|}{\|\nu\|\|\omega\|} \hspace{11pt}\text{in the real case}, \\[3pt]
\frac{|\Prod{\nu,\omega}|^2}{\|\nu\|^2\|\omega\|^2} \,\text{ in the complex case}.
\end{cases}  \]
\end{theorem}
\begin{proof}
As $\dim \Lambda^p V = \dim \Lambda^p W =1$, it is enough to consider unit $p$-blades $\nu=e_1\wedge\ldots\wedge e_p$ and $\omega=f_1\wedge\ldots\wedge f_p$ formed with principal vectors of $V$ and $W$. Then \eqref{eq:eifj} gives $\Prod{\nu,\omega}= \sigma_1\cdot\ldots\cdot\sigma_p$, and the result follows from \eqref{eq:pi sigma} and \autoref{pr:projection factors}.
\end{proof}

So, for subspaces of same dimension, $\pi_{V,W}$ is related to the inner product of blades.
In \autoref{sc:Properties of projection factors}, similar formulas show that for different dimensions it is related to the interior product (\autoref{pr:properties pi}\,\ref{it:interior product}), and $\pi_{V,W^\perp}$ is related to the exterior product (\autoref{pr:properties pi perp}\,\ref{it:pi perp exterior product}).

A line $L$ forms with a subspace $W$ a single principal angle, which is simply called \emph{the angle} $\theta_{L,W}$ between them. In this case, \autoref{pr:projection factors} turns \eqref{eq:principal factor angle} into
\begin{equation}\label{eq:pi line}
\pi_{L,W}=
\begin{cases}
\cos \theta_{L,W} \ \,\text{ in the real case,} \\
\cos^2 \theta_{L,W} \text{ in the complex case.}
\end{cases}
\end{equation}
Extending this to arbitrary subspaces requires an appropriate concept of angle between subspaces, as in high dimensions there are several distinct ones, such as minimal angle  \cite{Dixmier1949}, Friedrichs angle \cite{Friedrichs1937}, and others. The reason for such diversity is that the whole list of principal angles is needed to fully describe the relative position of two subspaces \cite{Wong1967}.

\begin{definition}
Let $V,W\subset X$ be nonzero subspaces, $p=\dim V$ and $q=\dim W$.
The \emph{Grassmann angle} $\Theta_{V,W}\in[0,\frac{\pi}{2}]$ of $V$ with $W$ is 
\[ \Theta_{V,W}=\begin{cases}
\arccos(\cos\theta_1\cdot\ldots\cdot\cos\theta_p) \ \text{ if } p\leq q, \\
\frac{\pi}{2} \hspace{97pt} \text{ if } p> q,
\end{cases} \]
where the $\theta_i$'s are the principal angles of $V$ and $W$.
We also define $\Theta_{\{0\},\{0\}}=\Theta_{\{0\},W}=0$ and $\Theta_{V,\{0\}}=\frac{\pi}{2}$.
\end{definition}

This angle was introduced in \cite{Mandolesi_Grassmann}, unifying and extending other angle concepts found in the literature \cite{Gluck1967,Gunawan2005,Hitzer2010a,Jiang1996}. It has many interesting properties, and some strange features, which derive from the fact that $\Theta_{V,W}$ is actually an angle in the exterior power $\Lambda^p X$, where $p=\dim V$, between the line $\Lambda^p V$ and the subspace\footnote{If $\dim W<p$ then $\Lambda^p W=\{0\}$, in which case the angle is defined as being $\frac{\pi}{2}$.} $\Lambda^p W$, 
\begin{equation}\label{eq:Theta exterior power}
\Theta_{V,W}=\theta_{\Lambda^p V,\Lambda^p W}.
\end{equation}

\begin{theorem}\label{pr:projection Grassmann angle}
Given any subspaces $V,W\subset X$,
\[ \pi_{V,W}=
\begin{cases}
\cos \Theta_{V,W} \ \,\text{ in the real case,} \\
\cos^2 \Theta_{V,W} \text{ in the complex case.}
\end{cases} \]
\end{theorem}
\begin{proof}
Follows from \eqref{eq:principal factor angle} and \autoref{pr:projection factors}.

\end{proof}

This lets us get many properties of projection factors (see \autoref{sc:Properties of projection factors}) from results about Grassmann angles.
For example, \eqref{eq:pi line} and \eqref{eq:Theta exterior power} give
\[ \pi_{V,W}=\pi_{\Lambda^p V,\Lambda^p W}. \]
Thus the factor by which top dimensional volumes in $V$ shrink when projecting to $W$ is the same by which lengths (areas, in the complex case) in the line $\Lambda^p V$ contract when projecting to $\Lambda^p W$. 

From formulas for computing $\Theta_{V,W}$ given in \cite{Mandolesi_Products}, we get the following ones, which generalize \autoref{pr:formula orthonormal bases} for arbitrary bases.

\begin{proposition}\label{pr:formula general bases}
Given bases $(v_1,\ldots,v_p)$ of $V$ and $(w_1,\ldots,w_q)$ of $W$, let $A=\big(\Prod{w_i,w_j}\big)$, $B=\big(\Prod{w_i,v_j}\big)$ and $D=\big(\Prod{v_i,v_j}\big)$. Then
\[ \pi_{V,W}=
\begin{cases}
\sqrt{\frac{\det(B^T \! A^{-1}B)}{\det D}}  \ \text{ in the real case,} \\[5pt]
\ \ \frac{\det(\bar{B}^T \! A^{-1}B)}{\det D} \ \ \text{ in the complex case.}
\end{cases} \]
If $\dim V=\dim W$ then
\[ \pi_{V,W}=
\begin{cases}
\frac{\left|\det B \right|}{\sqrt{\det A}\,\cdot\sqrt{\det D}} \ \,\text{ in the real case,} \\[5pt]
\ \ \frac{\left|\det B \right|^2}{\det A\,\cdot\,\det D} \ \ \ \text{ in the complex case.}
\end{cases} \]
\end{proposition}

\section{Generalized Pythagorean theorems}\label{sc:Pythagorean thms}

Using projection factors, we easily prove some known generalizations of the Py\-thag\-o\-re\-an theorem, while also extending them to the complex case.

\begin{definition}
An \emph{orthogonal partition} of $X$ is a collection $\{V_1,\ldots,V_k\}$ of mutually orthogonal subspaces such that $X=V_1\oplus\ldots\oplus V_k$.
\end{definition}

\begin{proposition}\label{pr:sum proj factors line}
For any line $L\subset X$ and any orthogonal partition $X=V_1\oplus\cdots\oplus V_k$,
\begin{itemize}
\item $\sum_{j=1}^k \pi_{L,V_j}^2 = 1$ in the real case, 
\item $\sum_{j=1}^k \pi_{L,V_j} = 1$ in the complex case.
\end{itemize}
\end{proposition}
\begin{proof}
Follows from \eqref{eq:norm projections} and \autoref{pr:Pv pi}, with any nonzero $v\in L$.
\end{proof}

\begin{example}
If $L$ is the complex line of a quantum state, and the $V_j$'s are the eigenspaces of an observable, then the $\pi_{L,V_j}$'s give the quantum probabilities, as noted in \autoref{ex:quantum probability}. So the complex case above corresponds to the condition of unit total probability.
\end{example}

This proposition leads to a generalized Pythagorean theorem relating the squared 1-dimensional measure of a set in a real line to the squared measures of its projections on the subspaces of a partition, and a complex version with non-squared 2-dimensional measures.

\begin{theorem}[Pythagorean theorem for lines]\label{th:Pythagorean lines}
Let $L\subset X$ be a line and $X=V_1\oplus\cdots\oplus V_k$ be an orthogonal partition.
Given any Lebesgue measurable set $S\subset L$, let $S_j$ be its orthogonal projection on $V_j$. Then
\begin{itemize}
\item $|S|_1^2 = \sum_{j=1}^k |S_j|_1^2$ \ in the real case,
\item $|S|_2 = \sum_{j=1}^k |S_j|_2$ \ in the complex case.
\end{itemize} 
\end{theorem}

\begin{example}
\autoref{fig:realpythagorean} illustrates the real case with $S$ consisting of 2 colinear segments: its squared total length is the sum of the squared total lengths of its orthogonal projections $S_x, S_y$ and $S_z$ on the axes. The same would hold if $S$ were a more complicated measurable subset of a line.
\begin{figure} [h]
\centering
\includegraphics[width=.4\linewidth]{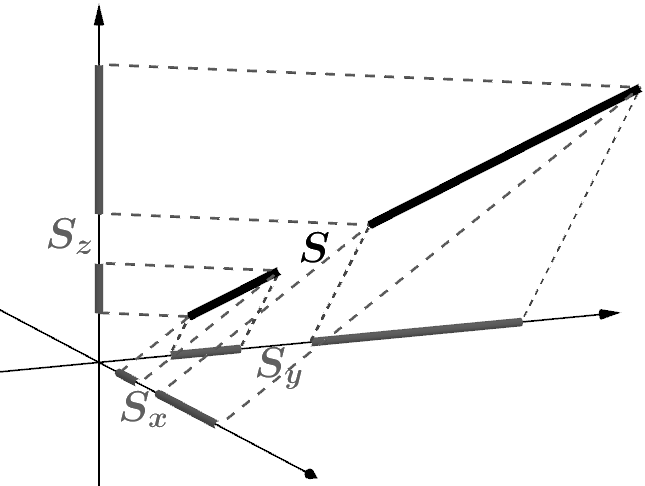}
\caption{Real Pythagorean Theorem, $|S|^2_1=|S_x|^2_1+|S_y|^2_1+|S_z|^2_1$}
\label{fig:realpythagorean}
\end{figure}
\end{example}

\begin{example}
Given $\C$-orthogonal unit vectors $v_1,v_2\in X$, and $c_1,c_2\in\C$ with $|c_1|^2+|c_2|^2=1$, let $v=c_1v_1+c_2v_2$. As \autoref{pr:product projection factor} gives $\pi_{\C v,\C v_1}=|c_1|^2$ and $\pi_{\C v,\C v_2}=|c_2|^2$, any area $A$ in $\C v$ projects to $A_1=|c_1|^2\!\cdot\! A$ in $\C v_1$ and to $A_2=|c_2|^2\!\cdot\! A$ in $\C v_2$, with $A=A_1+A_2$ (\autoref{fig:complexpythagorean}).
\begin{figure} [h]
\centering
\includegraphics[width=.4\linewidth]{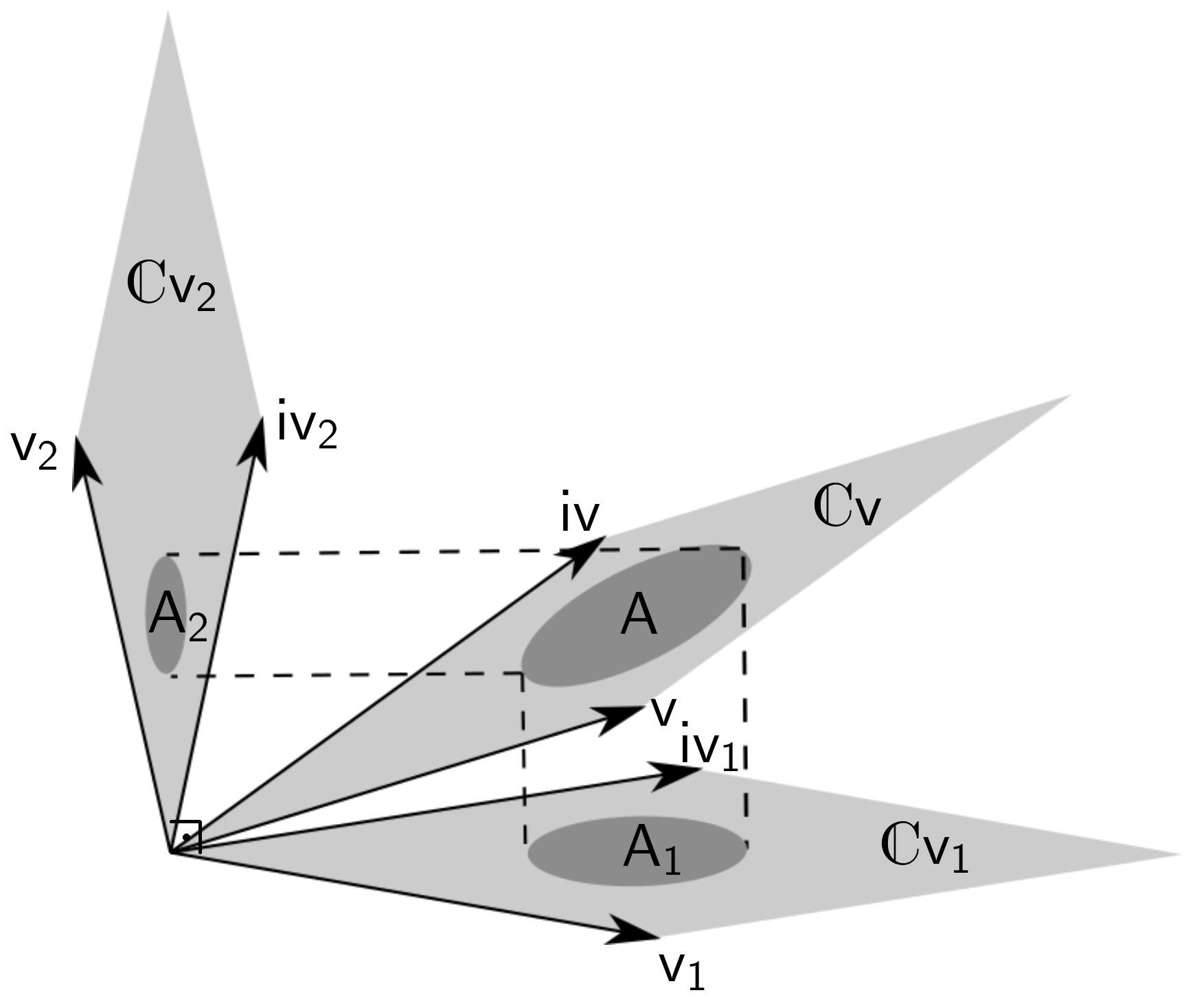}
\caption{Complex Pythagorean Theorem, $A=A_1+A_2$}
\label{fig:complexpythagorean}
\end{figure}

Similarities of this example with quantum theory are used in \cite{Mandolesi_Born} to obtain the Born rule in Everettian quantum mechanics.
\end{example}

Generalized Pythagorean theorems for subspaces other than lines involve projections on coordinate subspaces instead of a partition.

Let $\beta=(v_1,\ldots,v_n)$ be a basis of $X$, and $1\leq q\leq n$. For each multi-index $I =(i_1,\ldots,i_q)$, with $1\leq i_1 < \ldots<i_q\leq n$, we write $V_I = \Span(v_{i_1},\ldots,v_{i_q})$	and $\nu_I = v_{i_1}\wedge\ldots\wedge v_{i_q}\in \Lambda^q V_I$. The $\binom{n}{q}$ subspaces $V_I$ are the \emph{$q$-dimensional coordinate subspaces} of $\beta$, and the $\nu_I$'s constitute a basis of $\Lambda^q X$, which is orthonormal if $\beta$ is orthonormal.

\begin{proposition}\label{pr:projection factor coordinate subspaces}
If $V\subset X$ is a subspace and $p=\dim V$ then 
\begin{itemize}
\item $\sum_I \pi_{V,V_I}^2 = 1$ in the real case,
\item $\sum_I \pi_{V,V_I} = 1$ in the complex case,
\end{itemize}
where the sums run over all $p$-dimensional coordinate subspaces $V_I$ of an orthogonal basis of $X$.
\end{proposition}
\begin{proof}
Without loss of generality, we can assume the basis is orthonormal.
Let $\nu\in\Lambda^p V$ with $\|\nu\|=1$. 
Then $\sum_I |\Prod{\nu,\nu_I}|^2 = 1$, and the result follows from \autoref{pr:pi blades}.
\end{proof}

This gives another Pythagorean theorem, relating the squared $p$-di\-men\-sion\-al measure of a set in a real $p$-dimensional subspace to the squared measures of its projections on $p$-dimensional coordinate subspaces, and also a complex version with non-squared $2p$-dimensional measures.

\begin{theorem}[Pythagorean theorem for subspaces]\label{th:Pythagorean subspaces}
Let $V\subset X$ be a subspace and $p=\dim V$. For any Lebesgue measurable set $S\subset V$,
\begin{itemize}
\item $|S|_p^2 = \sum_I |S_I|_p^2$ \ \ in the real case,
\item $|S|_{2p} = \sum_I |S_I|_{2p}$ in the complex case,
\end{itemize}
where the sums run over the orthogonal projections $S_I$ of $S$ on all $p$-di\-men\-sion\-al coordinate subspaces of an orthogonal basis of $X$.
\end{theorem}

The real case gives, in particular, the following known generalizations of the Pythagorean theorem.

\begin{example}
A tetrahedron $OABC$ is \emph{trirectangular} at $O$ if all edges at this vertex are orthogonal to each other. De Gua's Theorem (1783) \cite{M.1931} says the squared area of the face opposite such vertex equals the sum of the squared areas of the other faces (\autoref{fig:tetraedro-edit}). It extends to higher dimensional simplices \cite{Amir-Moez1996,Alvarez1997,Cho1991,Donchian1935,Yeng1990}.
\begin{figure}[h]
\centering
\includegraphics[width=0.3\linewidth]{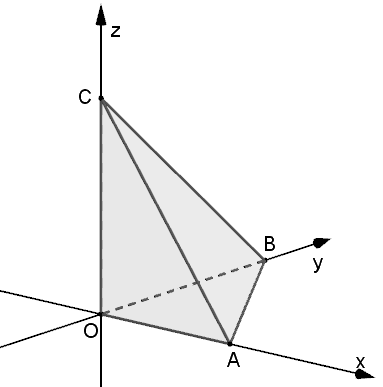}
\caption{$\area(ABC)^2=\area(OAB)^2+\area(OAC)^2+\area(OBC)^2$}
\label{fig:tetraedro-edit}
\end{figure}
\end{example}

\begin{example}
Since the XVIII century, it is known that the square of a planar area is the sum of the squares of its orthogonal projections on 3 mutually perpendicular planes \cite[p.450]{Eves1990}. Conant and Beyer \cite{Conant1974} have extended this to measurable sets (\autoref{fig:projecao_snowflake}) and any real dimension. They also tried to generalize it to the complex case, but the complex measure they used led to anomalous results.
We note that much of their work can be replaced by \autoref{rm:pi independent of S}.

\begin{figure}[h]
\centering
\includegraphics[width=0.5\linewidth]{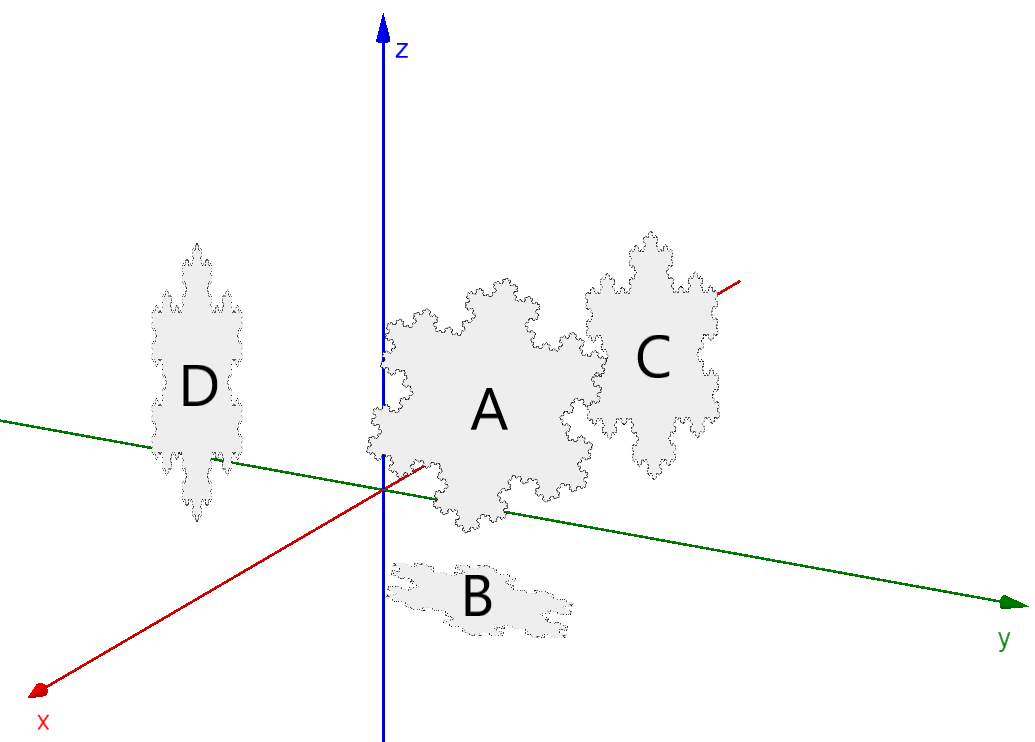}
\caption{Projections of a Koch snowflake, $|A|_2^2=|B|_2^2+|C|_2^2+|D|_2^2$}
\label{fig:projecao_snowflake}
\end{figure}
\end{example}

Yet another Pythagorean theorem, for projections on coordinate subspaces of a different dimension than $V$, can be obtained using \autoref{pr:projection Grassmann angle} and identities for Grassmann angles proven in \cite{Mandolesi_Products}.

\begin{proposition}
Let $V\subset X$ be a subspace, $p=\dim V$, and  $1\leq q\leq n=\dim X$. 
Then we have the following, where the sums run over all $q$-dimensional coordinate subspaces $V_I$ of an orthogonal basis of $X$.
\begin{enumerate}[i)]
\item If $p\leq q$ then 
\begin{itemize}
\item $\sum_I \pi_{V,V_I}^2 = \binom{n-p}{n-q}$ in the real case,
\item $\sum_I \pi_{V,V_I} = \binom{n-p}{n-q}$ in the complex case.
\end{itemize}
\item If $p>q$ then 
\begin{itemize}
\item $\sum_I \pi_{V_I,V}^2 = \binom{p}{q}$ in the real case,
\item $\sum_I \pi_{V_I,V} = \binom{p}{q}$ in the complex case.
\end{itemize}
\end{enumerate}
\end{proposition}

Only the case $p\leq q$ gives a Pythagorean theorem.

\begin{theorem}
Let $V\subset X$ be a subspace, $S\subset V$ be a Lebesgue measurable set, and $\dim V=p\leq q\leq n=\dim X$. Then 
\begin{itemize}
\item $|S|_p^2 = \binom{n-p}{n-q}^{-1}\cdot \sum_I |S_I|_p^2$ \ \ in the real case,
\item $|S|_{2p} = \binom{n-p}{n-q}^{-1}\cdot \sum_I |S_I|_{2p}$ in the complex case,
\end{itemize}
where the sums run over the orthogonal projections $S_I$ of $S$ on all $q$-di\-men\-sion\-al coordinate subspaces of an orthogonal basis of $X$.
\end{theorem}

The real case corresponds, when $S$ is a parallelotope, to a result of \cite{Drucker2015}. 
The binomial coefficients appear because each projection on a $q$-dimensional coordinate subspace can be further decomposed into the $p$-dimensional coordinate subspaces it contains. And each of these belongs to $\binom{n-p}{q-p}=\binom{n-p}{n-q}$ of the $q$-dimensional ones.

\begin{example}
Given a line segment of length $L$ in $\R^3$, let $L_{xy}, L_{xz}$ and $L_{yz}$ be the lengths of its orthogonal projections on the coordinate planes (\autoref{fig:line_planes}). Then $L^2=(L_{xy}^2+L_{xz}^2+L_{yz}^2)/2$.

\begin{figure}[h]
\centering
\includegraphics[width=0.4\linewidth]{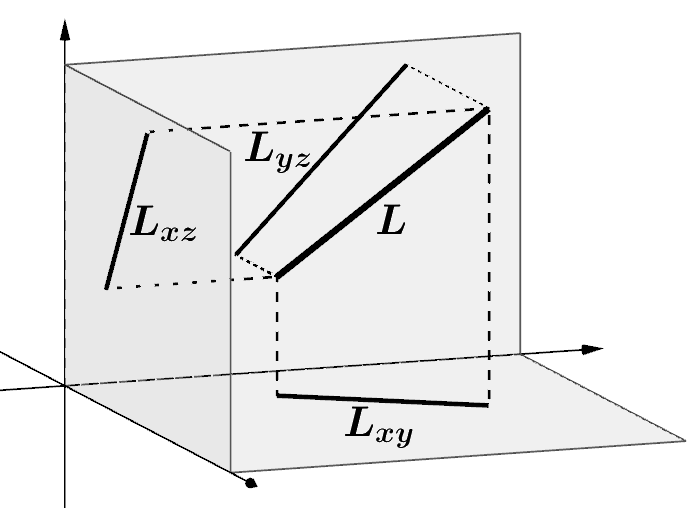}
\caption{$L^2=(L_{xy}^2+L_{xz}^2+L_{yz}^2)/2$}
\label{fig:line_planes}
\end{figure}

Using \autoref{th:Pythagorean lines} to relate $L$, $L_{xy}$, $L_{xz}$ and $L_{yz}$ to the orthogonal projections of these segments on the coordinate axes, one can easily see the reason for the factor $\frac{1}{2}$ in this example.
\end{example}

\section{Final remarks}\label{sc:epilogue}

For simplicity, we have worked with linear subspaces of a finite dimensional inner product space $X$, but our results obviously extend to affine subspaces, and $X$ can be an infinite dimensional Hilbert space (however, $V$ must be finite dimensional and $W$ must be closed for $\pi_{V,W}$ to be defined).

That measures are not squared in the complex Pythagorean theorems is unusual, but the reason is clearly that each complex dimension corresponds to 2 real ones, both contracting by the same factor. 

The same dimensional consideration suggests quaternionic versions should involve square roots of measures. Indeed, the tools used to get to \autoref{th:Pythagorean lines} also work in a complete quaternionic inner product space \cite{Istratescu1987,Rodman2014}. And essentially the same arguments can be used to show that, in such space, the square root of the 4-dimensional Lebesgue measure of a measurable set in a quaternionic line is the sum of the square roots of the measures of its orthogonal projections on the subspaces of an orthogonal partition. It would be interesting to know whether \autoref{th:Pythagorean subspaces} can also be extended to the quaternionic case.

The following examples show the complex case is actually simpler when calculations are carried out: there are less coordinate subspaces, and the measures add up in a simpler way. The real case leads to a messier calculation, in which it is almost surprising that the terms that appear when we expand the squares of the $|S_{ij}|_2$'s can be recombined into just the square of $|S|_2$. 

\begin{example}\label{ex:real subspace}
In $\R^4$, let $v=(a,b,c,d)\neq 0$, $u=(-b,a,-d,c)$ and $V=\Span(v,u)$.
If $\{v_1,v_2,v_3,v_4\}$ is the canonical basis of $\R^4$, let $P_{ij}$ be the orthogonal projection onto the coordinate subspace $V_{ij}=\Span(v_i,v_j)$. The projections $S_{ij}$ of the square $S=[v,u]$ on the $V_{ij}$'s have measures $|S_{ij}|_2=\|P_{ij}v\wedge P_{ij}u\|$ given by
\begin{align*}
|S_{12}|_2&=a^2+b^2, &
|S_{13}|_2=|S_{24}|_2&=|bc-ad|, \\
|S_{34}|_2&=c^2+d^2, &
|S_{14}|_2=|S_{23}|_2&=|ac+bd|. 
\end{align*}
A calculation shows the sum of the squares of these six measures is equal to $(a^2+b^2+c^2+d^2)^2=|S|_2^2$, in accordance with \autoref{th:Pythagorean subspaces}. 
\end{example}

\begin{example}\label{ex:complex line}
	The previous example can be reframed in complex terms. Identifying $\C^2$ with $\R^4$ we have $v=(a+ib,c+id)$, $u=iv$ and $V=\C v$.
	Of the $V_{ij}$'s, only $V_{12}$ and $V_{34}$ are complex subspaces (invariant under multiplication by $\im$), corresponding to the complex coordinate lines of the canonical basis $\{(1,0),(0,1)\}$ of $\C^2$. One can immediately see that, in accordance with \autoref{th:Pythagorean lines}, $|S_{12}|_2+|S_{34}|_2=a^2+b^2+c^2+d^2=|S|_2$.
\end{example}

The many properties projection factors have (including those in the appendix), and the ease with which the generalized Pythagorean theorems were obtained using them, suggest they should have other applications. 

As noted, they have unexpected connections with quantum theory. In \cite{Mandolesi_Born} we show the complex case of \autoref{th:Pythagorean lines} can be used, with some extra physical assumptions, to obtain the Born rule and solve the probability problem of Everettian quantum mechanics, with projection factors actually playing the role of quantum probabilities.

\appendix

\section{Properties of projection factors}\label{sc:Properties of projection factors}

Many properties of projection factors are listed below. Some can be easily obtained from the definition and results presented in this article, while others follow from \autoref{pr:projection Grassmann angle}  and properties of Grassmann angles proven in \cite{Mandolesi_Grassmann,Mandolesi_Products}.

\begin{proposition}\label{pr:properties pi}
Let $V,W\subset X$ be nonzero subspaces, $p=\dim V$, $q=\dim W$, and $P:V\rightarrow W$ be the orthogonal projection.
\begin{enumerate}[i)]
\item $\pi_{\{0\},\{0\}}=\pi_{\{0\},W}=1$ and $\pi_{V,\{0\}}=0$.
\item $\pi_{V,W}=0 \ \Leftrightarrow\ V\cap W^\perp\neq\{0\}$.
\item $\pi_{V,W}=1 \ \Leftrightarrow\ V\subset W$.
\item $\pi_{V,W}=\pi_{V,P(V)}$.
\item If $\dim V=\dim W$ then $\pi_{V,W}=\pi_{W,V}$.
\item In the complex case (with $W$ complex),  for any  $v\in X$ and any nonzero $u\in \C v$ we have $\pi_{\C v,W}=\pi^2_{\R v,W}$ and $\pi_{\R u,W} = \pi_{\R v,W}$.
\item If $V'$ and $W'$ are the orthogonal complements of $V\cap W$ in $V$ and $W$, respectively, then $\pi_{V,W}=\pi_{V',W'}$.
\item If $V=\bigoplus_i V_i$, with the $V_i$'s being mutually orthogonal subspaces spanned by principal vectors of $V$ w.r.t. $W$, then $\pi_{V,W} = \prod_i \pi_{V_i,W}$. 
\item If $V=A\oplus B$, with $A\perp B$,  then
\begin{equation*}
\pi_{V,W} = \pi_{A,W}\cdot \pi_{B,W}\cdot \pi_{P(A),P(B)^\perp}.
\end{equation*}
\item For any subspace $U\subset W$, $\pi_{V,U} = \pi_{V,W}\cdot\pi_{P(V),U}$.
\item $\pi_{V,W}\leq \pi_m$, where $\pi_m$ is the smallest principal projection factor of $V$ and $W$.
\item $\pi_{V,W'}\leq \pi_{V,W}$ for any subspace $W'\subset W$, with equality if, and only if, $V\cap W^\perp\neq\{0\}$ or $P(V)\subset W'$.
\item $\pi_{V',W}\geq\pi_{V,W}$ for any subspace $V'\subset V$, with equality if, and only if, $V'\cap W^\perp\neq\{0\}$ or $V\cap V'^\perp\subset W$.
\item Let $\mathcal{P}\nu=P v_1\wedge\ldots\wedge P v_p$ be the orthogonal projection on $\Lambda^p W$ of  a blade $\nu=v_1\wedge\ldots\wedge v_p\in\Lambda^p V$. Then \label{it:projection blade}
\begin{itemize}
\item $\|\mathcal{P}\nu\| = \|\nu\|\cdot \pi_{V,W}$ \hspace{9pt} in the real case,
\item $\|\mathcal{P}\nu\|^2 = \|\nu\|^2\cdot \pi_{V,W}$  \, in the complex case.
\end{itemize}
\item Given nonzero blades $\nu\in\Lambda^p V$ and $\omega\in\Lambda^q W$,\label{it:interior product}
\[
\pi_{V,W} = \begin{cases}
\hspace{2pt}\frac{\|\nu\intprod\omega\|}{\|\nu\|\|\omega\|}  \ \ \,\text{ in the real case,} \\[5pt]
\frac{\|\nu\intprod\omega\|^2}{\|\nu\|^2\|\omega\|^2}\ \text{ in the complex case,} 
\end{cases}
\]
where $\intprod$ is the interior product of blades \cite{Mandolesi_Products}.
\item If $U\subset V$ is a subspace and $r=\dim U$,
\[
\pi_{U,W} = \begin{cases}
\sqrt{\sum_{I} \pi^2_{U,V_I} \cdot \pi^2_{V_I,W}}\ \ \,\text{ in the real case,} \\[5pt]
\hspace{9pt}\sum_{I} \pi_{U,V_I} \cdot \pi_{V_I,W} \hspace{8pt} \text{ in the complex case.} 
\end{cases}
\]
where the sums run over all $r$-dimensional coordinate subspaces $V_I$ of a principal basis of $V$ with respect to $W$.
\end{enumerate}
\end{proposition}

There are also several properties for projection factors with the orthogonal complement of a subspace.

\begin{proposition}\label{pr:properties pi perp}
Let $V,W\subset X$ be nonzero subspaces, $p=\dim V$, $q=\dim W$ and $m=\min\{p,q\}$.
\begin{enumerate}[i)]
\item $\pi_{V,W^\perp} = \pi_{W,V^\perp}$.
\item $\pi_{V^\perp,W^\perp} = \pi_{W,V}$.
\item If $\pi_1,\ldots,\pi_m$ are the principal projection factors of $V$ and $W$ then
\[ \pi_{V,W^\perp}=
\begin{cases}
\prod_{i=1}^m \sqrt{1-\pi_i^2} \,\text{ in the real case,} \\[5pt]
\prod_{i=1}^m (1-\pi_i) \ \ \text{ in the complex case.}
\end{cases} \]
\item For $\zeta(V,W) = \begin{cases}
\pi^2_{V,W} + \pi^2_{V,W^\perp} \ \text{(real case)} \\
\pi_{V,W} + \pi_{V,W^\perp} \text{ (complex case)}
\end{cases}$ we have:
\begin{itemize}
\item $0\leq \zeta(V,W) \leq 1$;
\item $\zeta(V,W) = 1\ \Leftrightarrow\ \dim V=1$, or $V\subset W$, or $V\perp W$;
\item $\zeta(V,W)  = 0\ \Leftrightarrow\ V\cap W\neq\{0\}$ and $V\cap W^\perp\neq\{0\}$.
\end{itemize}
\item Given nonzero blades $\nu\in\Lambda^p V$ and $\omega\in\Lambda^q W$,\label{it:pi perp exterior product}
\[
\pi_{V,W^\perp} = \begin{cases}
\hspace{2pt}\frac{\|\nu\wedge\omega\|}{\|\nu\|\|\omega\|}  \ \ \,\text{ in the real case,} \\[5pt]
\frac{\|\nu\wedge\omega\|^2}{\|\nu\|^2\|\omega\|^2}\ \text{ in the complex case.} 
\end{cases}
\]
\item Given any bases $(v_1,\ldots,v_p)$ of $V$ and $(w_1,\ldots,w_q)$ of $W$, let $A=\big(\Prod{w_i,w_j}\big)$, $B=\big(\Prod{w_i,v_j}\big)$, and $D=\big(\Prod{v_i,v_j}\big)$. Then
\[ \pi_{V,W^\perp}=
\begin{cases}
\sqrt{\frac{\det(A-BD^{-1}B^T )}{\det A}}  \ \,\text{ in the real case,} \\[5pt]
\ \ \frac{\det(A-BD^{-1}\bar{B}^T )}{\det A} \ \ \text{ in the complex case.}
\end{cases} \]
If both bases are orthonormal then
\[ \pi_{V,W^\perp}=
\begin{cases}
\sqrt{\det\left(\mathds{1}_{q\times q}-B B^T\right)}  \ \,\text{ in the real case,} \\
\ \ \det\left(\mathds{1}_{q\times q}-B\bar{B}^T\right) \ \, \text{ in the complex case.}
\end{cases} \]
\end{enumerate}
\end{proposition}

\providecommand{\bysame}{\leavevmode\hbox to3em{\hrulefill}\thinspace}
\providecommand{\MR}{\relax\ifhmode\unskip\space\fi MR }
\providecommand{\MRhref}[2]{%
  \href{http://www.ams.org/mathscinet-getitem?mr=#1}{#2}
}
\providecommand{\href}[2]{#2}

\end{document}